\documentclass[oneside]{amsart}
\usepackage{amsmath,amssymb,amsthm}
\usepackage{lscape}
\usepackage{multirow}
\usepackage{rotating}
\usepackage{graphicx}
\usepackage{hyperref}
\usepackage{hyperref}
\hypersetup{
    colorlinks=true,
    linkcolor=blue,
    filecolor=magenta,      
    urlcolor=cyan,
    citecolor=blue}

\numberwithin{equation}{section} 
\numberwithin{figure}{section} 
\theoremstyle{plain}
\newtheorem*{thm*}{Theorem}
\theoremstyle{plain}
\newtheorem{thm}{Theorem}[section]
\theoremstyle{definition}
\newtheorem{defin}[thm]{Definition}
\theoremstyle{plain}
\newtheorem{lema}[thm]{Lemma}
\theoremstyle{plain}
\newtheorem{prop}[thm]{Proposition}
\theoremstyle{plain}
\newtheorem{cor}[thm]{Corollary}
\theoremstyle{remark}
\newtheorem{obs}[thm]{Remark}
\theoremstyle{remark}
\newtheorem*{acknowledgement*}{Acknowledgement}
\usepackage[utf8]{inputenc}

\newcommand{\T}{{\mathcal T}}

\newcommand{\C}{{\mathcal C}}

\newcommand{\To}{\longrightarrow}

\newcommand{\tm}{\T M}

\newcommand{\TM}{\mathcal T\hspace{-1pt}M}

\begin{document}

\title[On Finsler surfaces with certain flag curvatures]{On Finsler surfaces with certain flag curvatures}

\author {Ebstam  H. Taha}

\dedicatory{ }
\address{Harish-Chandra Research Institute, Chhatnag Road, Jhunsi, Allahabad 211019, India}
\address{Department of Mathematics, Faculty of Science, Cairo University, Giza 12613, Egypt}
\email{ebtsam.taha@sci.cu.edu.eg, ebtsamtaha@hri.res.in}
\maketitle 
\vspace{-1 cm}
\begin{abstract}
In the present paper, we find out necessary and sufficient conditions for a Finsler surface $(M,F)$ to be Landsbregian in terms of the Berwald curvature $2$-forms. We study Finsler surfaces which satisfy some flag curvature $K$ conditions, viz., $V(K)=0,\,\,V(K)= -\mathcal{I}/F^2$ and $V(K)=-\mathcal{I}\,K,$ where $\mathcal{I}$ is the Cartan scalar. In order to do so, we investigate some geometric objects associated with the global Berwald distribution $\mathcal{D}:= \operatorname{span}\{S, H, V:=JH\}$ of a $2$-dimensional Finsler metrizable nonflat spray $S$. We obtain some classifications of such surfaces and show that under what hypothesis these surfaces turn to be Riemannian. The existence of a first integral for the geodesic flow in each case has some remarkable consequences concerning rigidity results. We prove that a Finsler surface with $V(K)= -\mathcal{I}/F^2$ and   either $S(K)=0$ or $S(\mathcal{J})=0$  is Riemannian. Further, a Finsler surface with $V(K)=-\mathcal{I}\,K$ and $S(K)=0$ is Riemannian.\\ 
\end{abstract}
\maketitle 

\maketitle 

\section{Introduction}

The Finsler surface has been studied, from both local and global point of views, in~\cite{{Berwald41}, {Bryant95},  {Bucatarubook},  {Finslersur},  {Ikeda}}. In~\cite{Berwald41}, the Berwald frame has been introduced locally by Berwald  for a
Finsler surface and used to characterize and classify projectively flat Finsler surfaces. Such a frame has been studied and utilized  in~\cite[\S 9.9.1]{Szilasibook}. The Berwald frame $\left( H,S,V:JH,\mathcal{C}\right)$ has been defined directly for an arbitrary nonflat spray $S$ and its properties applied to obtain information about the Finsler metrizability of the given spray $S$ in~\cite{BucataruEbtsam}.  The Berwald distribution $\mathcal{D}$ is a  $3$-dimensional  distribution defined by $\mathcal{D}= \operatorname{span}\{S, H, V\}$, whose integrability provides a candidate for the Finsler function $F$. \\[-0.3cm]

The classification of Finsler spaces which are Riemannian is one of the fundamental problems which is known as Finsler rigidity theory. The notion of Riemannian sectional curvature is extended to Finsler geometry as the so-called flag curvature $K$. Also, Finsler geometry has many geometric objectives besides $K$, like Berwald, Landsbreg and Cartan curvatures. In case of Finsler surface, the Cartan (or main) scalar $\mathcal{I}$, Landsbreg scalar $\mathcal{J}:=S(\mathcal{I})$ and flag curvature completely determine its geometry. It is known that whenever $\mathcal{I}=0$, the Finsler space is Riemannian~\cite{deicke}, while for $\mathcal{J}=0$, the Finsler surface is Landsbregian. Moreover, the Finsler surface is Berwaldian when $\mathcal{J}=0$ and $H(\mathcal{I})=0$. However, in case of $K =0$ or a constant, the Finsler manifold is not necessarily Riemannian~\cite{Berwald41, Bryant95}. In~\cite{Akbar}, Akbar-Zadeh proved that \lq \lq \textit{any compact Finsler surface with negative constant  flag curvature is Riemannian}". In addition, Finsler manifolds with $K$ is just a function on $M$ are called $K$-basic Finsler manifolds and Schur's Lemma shows that these manifolds with $n \geq 3$ have constant $K$. However, the $2$-dimensional case needs different treatment, see e.g.~\cite{Foulon2016}. It is worth mentioning that there is a strong relation between the Riemannian character of a Finsler surface and the absence of conjugate points, see for instance~\cite{Finslersur}. \\[-0.3cm]

The aim of the present paper is to study some rigidity problems in Finsler surfaces. In this direction, we generalize some results of~\cite{Foulon2016,Finslersur,Ikeda}. The Ricci scalar of a Finsler space is given by $\rho := K F^2$~\cite{Bucatarubook}, thus one can work with $K$ or equivalently $\rho$. We investigate the Finsler surfaces which satisfy certain flag curvature (Ricci scalar) conditions, namely, 
${\rm i)}\,V(\rho)=0\Longleftrightarrow V(K)=0$, ii) $V(\rho)= -\mathcal{I}\Longleftrightarrow V(K)= -\mathcal{I}/F^2$ and iii) $V(\rho)= -\mathcal{I}\,\rho \Longleftrightarrow V(K)= -\mathcal{I}\, K$. One can say that, i) represents the case of $K$-basic Finsler surfaces,  ii) can be considered as  a generalization of i) in the sense that the main scalar $\mathcal{I}$, in general, is a function on $\T M$ not on $M$ and  iii) is equivalent to  $S(\mathcal{J})=0$, by \eqref{Bianchi}, that is $\mathcal{J}$ is a first integral for the geodesic flow. It can be considered as a generalization of Landsberg surface. Finsler surfaces with $S(\mathcal{J})=0$ reduce to  Landsbergian or Riemannian ones under some conditions, see for example  Theorems \ref{thm1:s(J)=0} and \ref{thm:s(J)=0}.  \\[-0.3cm]

A key aspect of our work is the use of the  Berwald distribution $\mathcal{D}$ and its  Berwald connection. In this paper, we  find out, in details, some geometric objects related to the Berwald connection such as  the Bianchi identities which describe the relations between some derivatives of the three invariants  $\mathcal{I},\,\mathcal{J},\,K$ with respect to Berwald distribution elements. In addition, we compute  Berwald curvature, Berwald connection $1$-forms, curvature $2$-forms and Cartan's structural equations associated  with $\mathcal{D}$, etc. Then, we use these relations to study the Finsler surfaces which satisfy the  aforementioned curvature conditions. Thereby, we obtain some classifications of such Finsler  surfaces.\\[-0.3cm]

The paper is organized as follows. We recall some basic facts on the geometry
of sprays and Finsler manifolds in~\S 2. We find out some geometric quantities associated with the Berwald distribution using Berwald connection and Cartan's structural equations in~\S 3.  In~\S 4, we give a necessary and sufficient condition for a Finsler surface to be Landsbergian in terms of the Berwald curvature $2$-forms. In Landsbreg surfaces $(M,F)$, we prove that the function  $\mathcal{F}:=H(\mathcal{I})+V(\mathcal{J})$ is a first integral for the geodesic flow  if and only if $(M,F)$ is Riemannian. In~\S 5, we briefly redo some analysis of~\cite{Foulon2016} in a slightly different way. We show the existence of a first integral $\mathcal{F}$ for the geodesic flow of any Finsler space with $V(K) = 0$. In addition, we show how Berwald distribution and its associated geometric quantities can be used to derive some rigidity results for Finsler structure defined on compact surfaces. After that, we study the case  of  $V(\rho)=-\mathcal{I}$ and show that  when either the flag curvature is a first integral for the geodesic flow or  the Landsbreg scalar $\mathcal{J}$ is a first integral for the geodesic flow, the Finsler surface reduces to Riemannian one. Consequently, a Landsberg surface with $V(\rho)=-\mathcal{I}$  is Riemannian.  Finally,  we investigate the case of Finsler surfaces  with $V(K)=-\mathcal{I}\,K $ as a generalization of~\cite{Foulon2016}. In this case, we show that a compact Finsler surface of genus at least one without conjugate points is Riemannian. Finally, we prove that Finsler surface with $V(K)=-\mathcal{I}\,K$ and $S(K)=0$ is Riemannian.

\section{Preliminaries}
Now, we recall some definitions from geometry of sprays and Finsler spaces. We refer to~\cite{Bucatarubook, Szilasibook} for further reading.
Let $M$ be an $n$-dimensional manifold, $(TM,\pi_M,M)$ be its tangent bundle
and $(\T M,\pi,M)$ be the subbundle of nonzero tangent vectors.  Hereafter,
$x^i $ and $(x^i, y^i)$ denote the local coordinates on the base manifold $M$ and the
induced coordinates on $TM$, respectively. The vector $1$-form $J$ on $TM$ is defined,
locally, by $J = \frac{\partial}{\partial y^i} \otimes dx^i$ \footnote{From here onwards the Einstein summation convention is in place.} is called the almost-tangent structure of $T M$. The vertical vector field
$\C=y^i\frac{\partial}{\partial y^i}$ on $TM$ is known as  the
Liouville vector field. 

\begin{defin}
A vector field $S\in \mathfrak{X}(\T M)$ is said to be a spray on $M$ if $JS = \C$ and
$[\C, S] = S$. Locally, $S$ is given by \vspace{-0.3cm}
\begin{equation}
  \label{eq:spray}
  S = y^i \frac{\partial}{\partial x^i} - 2G^i\frac{\partial}{\partial y^i},
\end{equation}
where $G^i=G^i(x,y)$ are the \emph{spray coefficients} which are positive homogeneous of degree $2$ in  $y$.
\end{defin}

Each spray $S$ induces, using the Frolicher-Nijenhuis formalism, a canonical nonlinear connection $\Gamma := [J,S]$.  The existence of $\Gamma$ is equivalent to the existence of an $n$-dimensional distribution \[H : u \in \tm \longrightarrow H_u\in T_u(\tm)\] that is supplementary to the vertical distribution which is called the horizontal distribution. Thus,  we have 
\[T_u(\tm) = H_u(\tm) \oplus V_u(\tm), \quad \forall u \in \tm,\]
where the corresponding horizontal and vertical projectors, respectively, are given by
\begin{equation}
  \label{projectors}
    h:=\frac{1}{2}  (Id + [J,S]), \,\,\,\,\,\,            v:=\frac{1}{2}(Id - [J,S]).
\end{equation}
For every $Z\in  \mathfrak{X}(M)$,  $\mathcal{L}_Z$ and $ i_{Z}$ denote the Lie derivative  with respect to $Z$ and the interior product by $Z$, respectively. The differential of a function $h$ is denoted by $dh$. A vector $r$-form on $M$ is a skew-symmetric $C^\infty(M)$-linear map $T:(\mathfrak{X}(M))^{r}\longrightarrow \mathfrak{X}(M)$. A vector $r$-form $T$ defines two graded derivations  $d_T$ and $i_T$ of the Grassman algebra of $M$ such that
\begin{equation}\label{exter. derv.}
d_T:=[i_T,d]=i_T\circ d-(-1)^{r-1}d\circ i_T,
\end{equation} 
 \begin{equation}\label{interior derv.}
i_T h =0, \,\,\,\, i_T dh=dh\circ T,\,\,\,\,h\in C^\infty(M).
\end{equation} 
 
The Jacobi endomorphism induced by $S$ is defined by
\begin{equation}\label{Jacobi endo}
\Phi=v\circ [S,h]=R^i_{j}\,\frac{\partial}{\partial y^i}\otimes dx^j=\left(2\,\frac{\partial G^i}{\partial x^j}-S(G^i_j)-G^i_k \,G^k_j \right)\frac{\partial}{\partial y^i}\otimes dx^j.
\end{equation}

\begin{defin}\label{Def.Isotropic spray}
  A spray $S$ is called \emph{isotropic} if  the Jacobi endomorphism has the form
  \begin{displaymath}
    \Phi=\rho \, J-\beta\otimes C,
  \end{displaymath}
  where $\rho:=\operatorname{Tr}(\Phi)$ is called the Ricci scalar and $\beta$ is semibasic $1$-form. 
\end{defin}
\begin{defin}
A smooth Finsler structure on an $n$-dimensional manifold $M$ is a mapping  $$F: TM \To \mathbb{R} ,\vspace{-0.3cm}$$  such that{\em:}
 \begin{itemize}
    \item[(a)] $F$ strictly positive on $\T M$ and $F(x,y)=0$ if and only if $y=0$,
    \item[(b)]$F$ is positively homogenous of degree $1$ in the directional argument $y$, that is 
    $\mathcal{L}_{\mathcal{C}} F=F$,
    \item[(c)] The metric tensor $g_{ij}=\frac{1}{2} \frac{\partial ^2 \, F^2}{\partial y^i \,\partial y^j} $ has maximal rank on $\T M$.
 \end{itemize}
 The pair $(M,F)$ is called Finsler manifold.
 \end{defin}
 
\begin{defin}
  A spray $S$ on a manifold $M$ is called \emph{Finsler metrizable} if there
  exists a Finsler function $F$ such that the geodesic spray of $F$ is $S$.
\end{defin}
 
 \begin{defin} \label{iso Finsler def}
 The Finsler manifold  $(M,F)$ is said to be of scalar flag curvature if there exists a function $K\in C^\infty(\TM)$ such that the associated Jacobi endomorphism is given by
  \begin{displaymath}
    \Phi=K\,(F^2 J-Fd_JF\otimes C).
  \end{displaymath}
  Moreover, if $K$ is constant then   $(M,F)$  is called of constant flag curvature. 
\end{defin}
It follows that  a Finsler function $F$ of scalar flag curvature $K$, its geodesic spray $S$ is isotropic with Ricci scalar $\rho=K\, F^2$ and the semi-basic $1$-form $\beta =K\, F\,d_JF$. In fact, any $2$-dimensional spray is always isotropic~\cite[Corollary 8.3.11]{Szilasibook}. When $K$ is a constant, we say that $S$ has a constant flag
  curvature.   
 \begin{defin}
Let $\eta$ be a unit speed geodesic, i.e. $\eta '' = S \circ \eta '$, in a Finsler surface $(M,F)$. A vector field $Y \in \mathfrak{X}(M)$ along $\eta$ is called Jacobi field if it satisfies
\[ D_{X}  D_{X} Y + R(Y,X)X=0,\]
where $D_{X} Y$ is the covariant differentiation with respect to Chern-Rund connection of $Y$ along $X$ and
 $$R(u,v)v = K(u) \left[g_{u}(u,u)v -  g_{u}(u,v)u \right], \quad v,u \in \T_{x}M. $$
\end{defin}

\begin{defin}
A point $p \in M$ is said to be conjugate point to $q \in M$ along a geodesic $\eta$ if there exists a nonzero Jacobi field along  $\eta$ that vanishes at $p$ and $q$. A Finsler manifold $(M,F)$ is called without conjugate points if no geodesic in $(M,F)$ has conjugate points.
\end{defin}

\begin{obs}
The three invariants $\mathcal{I},\, \mathcal{J},\, K$ are functions on the unit tangent bundle (or indicatrix bundle, simply, indicatrix) $IM$~\cite{Berwald41}. It is worth noting that the main scalar $\mathcal{I}$ is a convex geometric invariant describes the shape of each unit tangent circle, the Landsbreg scalar $\mathcal{J}= S(\mathcal{I})$ is the rate of change of $\mathcal{I}$ along the geodesics and the flag $K$ is a variational invariant which measures the focusing of geodesics. The homogeneity of these invariants are given by \vspace{-0.1 cm}
\begin{equation}\label{inv. hom.}
\mathcal{C}(K)=0,\quad \mathcal{C}(\mathcal{I})=0,\quad \mathcal{C}(\mathcal{J})=\mathcal{J}.
\end{equation}
\end{obs}


Let us end this section by recalling some interesting rigidity results on Finsler surfaces.

\begin{thm}\label{Fsurconjresults}
Let $(M,F)$ be a compact Finsler surface of genus at least one without conjugate points. Then, the following are satisfied: 
\begin{enumerate}
\item {If $F$ is a Landsberg metric, then $(M,F)$ is Riemannian~\cite[Theorem 2]{Finslersur}.}
\item {If $S(K)=0$, then $(M,F)$  has constant flag curvature $K$~\cite[Proposition 7.1]{Finslersur}.}
\item {If  $(M,F)$ has $S(\mathcal{J})=0$ and $S(K)=0$, the Finsler surface has a constant flag curvature $K$ and it is Riemannian whenever $K$ nonvanishing~\cite[Lemma 7.3]{Finslersur}. }
\end{enumerate}
\end{thm}

\begin{thm}~\!\!\cite{Foulon2016} \label{Foulon result}
Let $(M,F)$ be a smooth compact connected  $K$-basic Finsler surface
without conjugate points and genus greater than one. Then, $(M,F)$ is Riemannian.
\end{thm}
\begin{thm}\cite{Foulon97} The flag curvature is a first integral for the geodesic flow, that is $S(K)=0$, if the Finsler
structure is reversible $C^3$ on $\T M$
and locally symmetric.
\end{thm}
\begin{thm}\cite{Ikeda} \label{Ikeda thm}
If $(M,F)$ be a connected Landsberg surface with $S(K)=0$, then  $(M,F)$ has a constant flag curvature and is Riemannian whenever $K$ nonvanishing.
\end{thm}

\section{Berwald frame and Berwald connection on Finsler surfaces}


\begin{defin}\cite{BucataruEbtsam} 
Let $(M,S)$ be a $2$-dimensional smooth manifold $M$ equipped with a nonflat spray $S$. Let $H \in \mathfrak{X}(\T M)$ be a positive $2$-homogeneous horizontal
vector field  such that $\beta(H)=0$, where $\beta$ is semi-basic $1$-form defined in Definition~\ref{Def.Isotropic spray}. The regular $3$-dimensional  distribution given by $\mathcal{D}= \operatorname{span}\{S, H, V=JH\}$ is called Berwald distribution. Berwald frame is a global frame on $\T M$ defined by $( H,S, V, \mathcal{C})$. 
\end{defin}
When the spray is Finsler metrizable, the semi-basic $1$-form is given by $\beta= K\,F\,d_{J}F$ and $\mathcal{D}$ is an integrable distribution.  Hereafter, we assume that $S$ is a nonflat spray metrizable by a Finsler function $F$. The nonflatness assumption of $S$  that we work with is equivalent to $K \neq 0$.

\begin{lema}\cite{BucataruEbtsam} \label{BE}
Let $S$ be the geodesic spray of a Finsler function $F$ and let $\mathcal{D}$ be its
Berwald distribution. Then, we have
\begin{equation}\label{V(F)=0}
 H(F^2)=0,\quad  V(F^2)=0, 
\end{equation}
\begin{equation}\label{Berwald_f2 SH} 
[S, H]= \rho V, 
\end{equation}
\begin{equation}\label{Berwald_f2 VS} 
[V,S] = H, 
\end{equation}
\begin{equation}\label{Berwald_f2 HV} 
[H,V]=S+\mathcal{I}\,H +\mathcal{J}\, V, \, \text{ where }\mathcal{J}= S(\mathcal{I}). 
\end{equation}
\end{lema}
\begin{defin}\cite[Theorem 3.2.1]{Bucatarubook}
The map $D:\mathfrak{X}(TM)\times\mathfrak{X}(TM)\rightarrow \mathfrak{X}(TM)$, is given by
\begin{equation}
D_XY=v[hX,vY]+h[vX,hY]+J[vX,\theta Y]+\theta[hX,JY],
\end{equation}\label{Berwald}
where $\theta$ is the adjoint structure defined by $\theta = h \circ [S,h]$. $D$ is a linear connection on $TM$ which is called Berwald connection.
\end{defin}

\begin{prop}\label{conn.}
Consider $S$ be the geodesic spray of a Finsler function $F$ and let $( H,S, V, {\mathcal C})$ be its
Berwald frame. Then, the nonvanishing components of Berwald connection are the following:
\[D_HH=\mathcal{J}\,H,\quad D_HV=\mathcal{J}\,V,\quad D_V\mathcal{C}=V,\quad D_VS=H,\quad D_VH=-S-\mathcal{I}\,H,\]
\[D_{\mathcal{C}}H=H,\quad D_{\mathcal{C}}V=V,\quad D_VV=-\mathcal{C}-\mathcal{I}\,V,\quad D_{\mathcal{C}}S=S,\quad D_{\mathcal{C}}\mathcal{C}=\mathcal{C}.\]
\end{prop}
\begin{proof}
It follows from the facts that $\{S, H\}$ are horizontal vector fields, $\{\mathcal C, V\}$ are vertical vector fields together with the following relations \[\theta h =0,\, \theta v = h,\, h^2 =h,\, v^2 =v,\, hv=vh=0\] 
and Lemma~\ref{BE}
 along with
\begin{equation}
 [\mathcal{C},H]=H,\quad [\mathcal{C},S]=S,\quad [\mathcal{C},V]=0.
\end{equation}
For example, \vspace{-0.15cm}
\[D_HH=v[hH,vH]+h[vH,hH]+J[vH,\theta H]+\theta[hH, JH]=\theta(S+\mathcal{I}\,H+\mathcal{J}\,V)=\mathcal{J}\,H.\vspace{-0.56cm}
\]
\end{proof}

\begin{prop}\label{cur. coff.}
Consider $S$ be a Finsler metrizable spray and let $(H,S, V, {\mathcal C})$ be its
Berwald frame. The only nonvanishing components of the Berwald curvatures are given by
\begin{enumerate}
\item{$hh$-curvature $R(S,H)S= -\rho \, H$,\,\,  $R(S,H)H=\{S(\mathcal{J})+\rho\,\mathcal{I}\}\,H+\rho\,S $, }
\item{$hv$-curvature $B(V,H)H=\mathcal{F}\,H-2\,\mathcal{J}\,S,\,$ where  $\,\mathcal{F}:=H(\mathcal{I})+V(\mathcal{J}) .$}
\end{enumerate}
Consequently, $R(S,H)\mathcal{C}= -\rho \, V,\,\,\quad R(S,H)V= \{S(\mathcal{J})+\rho\,\mathcal{I}\}\,V+\rho\,\mathcal{C}, \quad B(V,H)V=\mathcal{F}\,V-2\,\mathcal{J}\,\mathcal{C}.$
\end{prop}
\begin{proof}It follows from the definition, see \cite{Bucatarubook}, of curvature tensor $C$ associated  with a connection $D$, \vspace{-0.1cm}
\[C(X,Y)Z:=D_XD_YZ-D_YD_XZ-D_{[X,Y]}Z\] and the following properties of $D$ \vspace{-0.1cm}
\[D_{fX}Y=fD_XY,\, \,\,\, D_XfY=fD_XY+X(f)Y,\] together with Proposition~\ref{conn.}, \eqref{Berwald_f2 SH}, \eqref{Berwald_f2 VS} and \eqref{Berwald_f2 HV}. Indeed, we get
\vspace{-0.1cm}
\begin{eqnarray*}
R(S,H)S&= &D_SD_HS-D_HD_SS-D_{[S,H]}S=-D_{\rho V}S=-\rho D_VS=-\rho H,\\
R(S,H)H&=& D_SD_H H-D_HD_S H-D_{[S,H]}H=D_S(\mathcal{J}\,H)-D_{\rho V}H\\
&=& \mathcal{J}\,D_S H +S(\mathcal{J})\,H -\rho \, D_V H=\{S(\mathcal{J})+\rho\,\mathcal{I}\}\,H+\rho\,S.
\end{eqnarray*}

Similarly, the nonvanishing component of $hv$-curvature is the following:
{\small{\begin{eqnarray*}
B(V,H)H&=&D_VD_HH-D_HD_VH-D_{[V,H]}H=D_V(\mathcal{J}\,H)-D_H(-S-\mathcal{I}H)+D_{(S+\mathcal{I}H+\mathcal{J}\,V)}H\\ &=&\mathcal{J}\,D_VH+V(\mathcal{J})\,H+D_HS+\mathcal{I}\,D_HH+H(\mathcal{I})\,H+D_SH+\mathcal{I}\,D_HH+\mathcal{J}\,D_VH \\&=&2\, \mathcal{J}\,D_{V}H + 2 \mathcal{I}\, D_{H}H + \{H(\mathcal{I})+V(\mathcal{J}) \}\, H  \\ &=&-2\,\mathcal{J}\,\{S +\mathcal{I}\,H \}+2\,\mathcal{I}\,\mathcal{J}\,H +\{H(\mathcal{I})+V(\mathcal{J}) \}\, H =-2\,\mathcal{J}\,S+\mathcal{F}\,H.
\end{eqnarray*}}}
The proof is completed by using the fact \cite{Bucatarubook},  $J(C(X,Y)Z)= C(X,Y)JZ, \quad X,Y,Z \in \mathfrak{X}(TM)$.  
\end{proof}

\begin{obs}
The function $\mathcal{F}=H(\mathcal{I})+V(\mathcal{J})$ is interesting and has some mysterious properties that proved to be effective and useful in characterization of Finsler surfaces cf.~\cite{Foulon2016}. This will be shown during presenting our results.
\end{obs}
\begin{thm}
Let $S$ be a Finsler metrizable spray and $( H, S, V, {\mathcal C})$ be its
Berwald frame. Then, we have 
\begin{equation}\label{Bianchi}
S(\mathcal{J})+V(\rho)+\mathcal{I}\,\rho=0,
\end{equation}
\begin{equation}\label{Bianchif}
S(\mathcal{F})+\mathcal{I}\,V(\rho)+V^2(\rho)=0, \, \text{ where } \,\,V^2(\rho) :=V(V(\rho)).
\end{equation}
\end{thm}
\begin{proof}
One can easily check that the only nonvanishing component of the torsion of Berwald connection $D$ is the $v(h)$-torsion, namely, $vT(S,H)=-\rho V$. Substituting by $X=S,\ Y=V,\ Z=H,\ W=V$ in the second Bianchi identity \cite{Bucatarubook}:$\sum\limits_{X,Y,Z}\{(D_XR)(Y,Z)+R(T(X,Y),Z)\}W =0, $ we obtain
\begin{eqnarray*}
0 &=& (D_SR)(V,H)V+(D_VR)(H,S)V+(D_HR)(S,V)V+R(T(S,V),H)V \\
&& +R(T(V,H),S)V+R(T(H,S),V)V \\
 &=& (D_SR)(V,H)V+(D_VR)(H,S)V+(D_HR)(S,V)V.
\end{eqnarray*}
Using $(D_XR)(Y,Z)W:=D_X(R(Y,Z)W)-R(D_XY,Z)W-R(Y,D_XZ)W-R(Y,Z)D_XW $
and substituting from Propositions~\ref{conn.} and~\ref{cur. coff.}, we get
\[
0=-2\left\{S(\mathcal{J})+V(\rho)+\mathcal{I}\,\rho\right\}\, \mathcal{C}+\left\{S(\mathcal{F})+\mathcal{I}\,V(\rho)+V^2(\rho)\right\}\, V. \]
The proof is completed from the fact that $\mathcal{C}$ and $V$ are independent vector fields which generate the vertical distribution.
\end{proof}
\begin{prop}
  Let $S$ be a Finsler metrizable spray and  $(H, S, V,\mathcal{C})$ be its associated Berwald frame with  corresponding dual basis $\{\eta^i, \,1\leq i\leq 4\}$. The Cartan's structural equations  are 
\begin{equation}\label{structure eq. 1}
d\eta^1=-\mathcal{I}\,\eta^1\wedge\eta^3+\eta^2\wedge\eta^3 ,
\end{equation}
\begin{equation}\label{structure eq. 2}
d\eta^2=-\eta^1\wedge\eta^3 ,
\end{equation}
\begin{equation}\label{structure eq. 3}
d\eta^3=\rho \,\eta^1\wedge \eta^2-\mathcal{J}\,\eta^1\wedge \eta^3.
\end{equation}
\end{prop} 

\begin{lema}\label{trace of Berwald cuvature}
Let $S$ be the geodesic spray of a Finsler function and $(H, S, V,\mathcal{C})$ be its associated Berwald frame with  corresponding dual basis $\{\eta^i, \,1\leq i\leq 4\}$. The trace of the $hv$ Berwald curvature is given by  $B(V,H)=\mathcal{F}$.
\end{lema}
 \begin{proof}
 It follows from the trace  definition, $B(X,Y):={\rm Tr}\left(Z\mapsto B(X,Z)Y\right)=\sum\limits_{\alpha=1}^{4}\eta^{\alpha}(B(X,X_{\alpha})Y).$ Indeed, taking into account Proposition~\ref{cur. coff.}, we get \vspace*{-0.1cm}
 \begin{eqnarray*}
 B(V,H)&=&\eta^1(B(V,H)H)+\eta^2(B(V,S)H)+\eta^3(B(V,V)H)+\eta^4(B(V,\mathcal{C})H)\\&=&\eta^1(\mathcal{F}\,H-2\mathcal{J}\,S)=\mathcal{F}.
 \end{eqnarray*}
\vspace*{-0.9cm}\[\qedhere\]
 \end{proof}
 \vspace*{-0.1cm}
In view of Proposition~\ref{cur. coff.} and Lemma~\ref{trace of Berwald cuvature}, we obtain: 
 \begin{cor}
 Let $S$ be the geodesic spray of a Finsler function $F$ and  $(H, S, V,\mathcal{C})$ be its  Berwald frame. Then, $(M,F)$ is Berwaldian if and only if it is Landsbergian and the trace $hv$-Berwald curvature vanishes identically.
\end{cor}
\begin{prop}
Let $S$ a Finsler metrizable spray and  $(H, S, V,\mathcal{C})$ be its  Berwald frame with  corresponding dual basis $\{\eta^i, \,1\leq i\leq 4\}$. The following matrix representation of the Berwald connection $1$-forms with respect to $(H, S, V,\mathcal{C})$ is given by
\[ \omega=[\omega_b^a]=\begin{bmatrix}
\mathcal{J}\,\eta^1-\mathcal{I}\,\eta^3+\eta^4 & -\eta^3 & 0 & 0\\
\eta^3 & \eta^4 & 0 & 0\\
0 & 0 & \mathcal{J}\,\eta^1-\mathcal{I}\,\eta^3+\eta^4 & -\eta^3\\
0 & 0 & \eta^3 & \eta^4
\end{bmatrix}.\]
 \end{prop}
 \begin{proof}
 It follows from Proposition~\ref{conn.}  along with the definition~\cite[\S 3.7]{Bucatarubook}
 \[\omega^a_b(X) =\eta^a(D_XX_b),\ 1\leq a, b\leq 4. \] 
 For example, \vspace{-0.59 cm}
\begin{eqnarray*}
\,\,\,\omega_1^1 &=&\eta^1(D_HH)\,\eta^1+\eta^1(D_SH)\,\eta^2+\eta^1(D_VH)\,\eta^3+\eta^1(D_{\mathcal{C}}H)\,\eta^4 \\
&= &\eta^1(\mathcal{J}\,H)\,\eta^1 +\eta^1(-S-\mathcal{I}\,H)\,\eta^3+\eta^1(H)\,\eta^4=\mathcal{J}\,\eta^1-\mathcal{I}\,\eta^3+\eta^4.
\end{eqnarray*} 
\vspace*{-0.9cm}
\[\qedhere\]
 \end{proof}
\begin{prop} 
Let $S$ a Finsler metrizable spray and  $(H, S, V,\mathcal{C})$ be its  Berwald frame with  corresponding dual basis $\{\eta^i, \,1\leq i\leq 4\}$, then the nonvanishing components of the Berwald curvature $2$-forms  are 
\begin{equation}\label{2-forms 11}
\Omega_1^1 = V(\rho)\,\eta^1\wedge\eta^2-\mathcal{F}\,\eta^1\wedge\eta^3, 
\end{equation}
\begin{equation}\label{2-forms 12}
\Omega_1^2 =\rho \,\eta^1\wedge\eta^2+2\,\mathcal{J}\,\eta^1\wedge\eta^3 ,
\end{equation}
\begin{equation}\label{2-forms 21}
 \Omega_2^1 = \rho\,\eta^1\wedge \eta^2  .\end{equation}

\end{prop}
\begin{proof}
The  curvature 2-forms   can be calculated through two ways: (i) using the Carant's structural equations, (ii) using the curvature tensor.
 In the first method, it can be calculated by  considering \eqref{structure eq. 1}, \eqref{structure eq. 2} and \eqref{structure eq. 3} along with  \[\Omega^i_j := -d\omega_j^i +\omega^l_j \wedge  \omega_l^i.\] 
In the second method, using Proposition~\ref{cur. coff.} in  $\Omega_j^i(X,Y):=\eta^i(R(X,Y)X_j)$
 and the fact that the matrix of the  curvature $2$-forms  has the following expression \cite{Bucatarubook} $\begin{pmatrix}\Omega_i^j & 0\\
0 & \Omega_i^j
\end{pmatrix}.$
In what follows, we compute $\Omega_1^1$ which is equal to $\Omega_3^3$ while $\Omega_1^2$ and $\Omega_2^1$ can be calculated by the same way,
\[\Omega_3^3(H,S)=\eta^3(R(H,S)V)=V(\rho),\,\,\quad\Omega_3^3(H,V)=\eta^3(R(H,V)V)=-\mathcal{F},\quad \,\Omega_3^3(H,\mathcal{C})=\eta^3(R(H,\mathcal{C})V)=0.\] This shows that the formula  \eqref{2-forms 11} is true.
\end{proof}

\section{Integrability condition for Landsberg surfaces}
Let us start this section with some important applications of the global Berwald frame using Berwald connection that we mentioned in the previous section.
\begin{lema}Let $S$ a Finsler metrizable spray. Then, the dual basis $\{\eta^i, \,1\leq i\leq 4\}$ of Berwald frame satisfies:
\begin{equation}\label{DJeta}
d_J\eta^1=0,\,\,\,\,\,\,\,\,\,\,\,\,\,d_J\eta^2=0,\,\,\,\,\,\,\,\,\,\,\,\,
d_J\eta^3=-d\eta^1 ,\,\,\,\,\,\,\,\,\,\,\,\,\,d_J\eta^4=-d\eta^2.
\end{equation}
\end{lema}
\begin{proof}Applying \eqref{exter. derv.} into the tangent structure $J$, we get\vspace{-0.14 cm} \begin{equation}\label{dj}
d_J\eta^r=i_Jd\eta^r-di_J\eta^r , \,\,\text{ for } 1\leq r\leq 4,
\end{equation} where $i_J$ is defined in \eqref{interior derv.}. The proof is completed by substituting  \eqref{structure eq. 1}, \eqref{structure eq. 2}  and   \eqref{structure eq. 3} into \eqref{dj}.
\end{proof}
 
\begin{prop}\label{dJ 2 forms}
Let $S$ a Finsler metrizable spray and  $(H, S, V,\mathcal{C})$ be its  Berwald frame with  corresponding dual basis $\{\eta^i, \,1\leq i\leq 4\}$, then
the Berwald curvature $2$-forms satisfy the following: 
\begin{equation}
d_J\Omega^3_3=0,\quad\,d_J\Omega_2^1=0,\quad\,d_J\Omega_1^2=-2\,\mathcal{J}\,\eta^1\wedge\eta^2\wedge\eta^3.
\end{equation}
\end{prop}
\begin{proof}
In view of \eqref{2-forms 11}, we have
\vspace{-0.1 cm}
 \begin{eqnarray*}\label{integrab}
d_J\Omega^3_3 &= & (d_JV(\rho))\wedge\eta^1\wedge \eta^2+V(\rho)\,d_J\eta^1\wedge \eta^2-V(\rho)\,\eta^1\wedge d_J\eta^2 \\
 && -(d_J\,\mathcal{F})\wedge\eta^1\wedge\eta^3-\mathcal{F}\,d_J\eta^1\wedge\eta^3+\mathcal{F}\,\eta^1\wedge d_J\eta^3.
\end{eqnarray*}
Considering \eqref{inv. hom.}, we get \vspace{-0.2cm}
\begin{eqnarray}\label{DJf} \nonumber
d_J\,\mathcal{F} &=&d \mathcal{F}\circ J =(d\mathcal{F}\circ JH)\eta^1+(d \mathcal{F}\circ  JS)\eta^2+(d \mathcal{F}\circ JV)\eta^3+(d \mathcal{F}\circ J\mathcal{C})\eta^4 \\ 
&= &(d\mathcal{F}\circ V)\eta^1+(d \mathcal{F} \circ\mathcal{C})\,\eta^2=V(\mathcal{F})\,\eta^1+\mathcal{C}(\mathcal{F})\,\eta^2=V(\mathcal{F})\eta^1+\mathcal{F}\,\eta^2 .
\end{eqnarray}
Similarly, $d_JV(\rho)=V^2(\rho)\,\eta^1+2V(\rho)\,\eta^2 .$
Taking into account \eqref{DJeta}, we obtain
\[d_J\Omega_3^3=-\mathcal{F}\,\eta^2\wedge\eta^1\wedge\eta^3+\mathcal{F}\,\eta^1\wedge(-d\eta^1)=0.\]
In the same way, the other component \eqref{2-forms 21} satisfies: 
\[d_J\Omega_2^1=(d_J\rho)\wedge\eta^1\wedge \eta^2+\rho \, (d_J\eta^1)\wedge\eta^2-\rho\,\eta^1\wedge d_J\eta^2=(V(\rho)\,\eta^2+2\rho\,\eta^2)\wedge\eta^1\wedge\eta^2=0.\]
Also, using \eqref{2-forms 12}, we get
\vspace{-0.15 cm}
\[d_J\Omega_1^2=d_J\left(\rho\,\eta^2\wedge \eta^1+2\,\mathcal{J}\,\eta^1\wedge\eta^3 \right)=-2\,\mathcal{J}\,\eta^1\wedge\eta^2\wedge\eta^3 .\]
\vspace*{-1cm}
\[\qedhere\]
\end{proof}

\begin{thm} Let $S$ a geodesic spray of a Finsler function $F$ and  $(H, S, V,\mathcal{C})$ be its  Berwald frame. Then, $(M,F)$ is  Landsbergian  if and only if the curvature $2$-form $\Omega_2^1$   is  $d_{J}$-closed.
\end{thm}
\begin{proof}
It follows directly from Proposition~\ref{dJ 2 forms} along with  the postulate $ \mathcal{J} =0$.
\end{proof}
The next result  classifies Landsberg surfaces in which  the trace $hv$-Berwald curvature $\mathcal{F}$ is a first integral of the geodesic flow.
\begin{thm}\label{Th:S(f)=0}
Let $S$ a geodesic spray of Landsbreg structure $F$ and  $(H, S, V,\mathcal{C})$ be its Berwald frame.  Then, the function $\mathcal{F}$ is a first integral of the geodesic flow, that is $S(\mathcal{F})=0$, if and only if $(M,F)$ is Riemannian. 
 \end{thm}
 \begin{proof}
 Let $(M,F)$ be a Landsberg surface, that is $S(\mathcal{I})=0$, thereby  $\mathcal{F}=H(\mathcal{I})$. 
 Applying the bracket $[S,H]$ to $\mathcal{I}$, we get \vspace{-0.2 cm}
 $$[S,H](\mathcal{I})=S(H(\mathcal{I})).$$
 But, by using the commutation formula~\eqref{Berwald_f2 SH},  we obtain $[S,H](\mathcal{I})= \rho V(\mathcal{I})$. Hence, we have
 $$S( \mathcal{F})=\rho V(\mathcal{I}).$$
 Putting $S(\mathcal{F})=0$,  we get $\rho=0$ or $V(\mathcal{I})=0$. But,  $\rho \neq 0 $ thus $V(\mathcal{I})=0$. Now,  $V(\mathcal{I})=0$ implies, by  \eqref{Berwald_f2 VS}, that \vspace{-0.1 cm}
 $$H(\mathcal{I})=[V,S](\mathcal{I})= V(S(\mathcal{I})) - S(V(\mathcal{I}))=0.$$
 That is, the horizontal covariant derivatives of $\mathcal{I}$  along both $S$ and $H$ vanish and this means that $\mathcal{I}$ is constant along the horizontal distribution. Hence, $(M,F)$ is Berwaldian. In fact, a Berwaldian with nonvanishing flag curvature is Riemannian \cite{Berwald41}. Therefore, $(M,F)$ is Riemannian.
 The converse follows directly, and this completes the proof. 
 \end{proof}
\section{Finsler surfaces with certain flag curvature}
We denote the Berwald connection $1$-form $\omega_1^1$ on the indicatrix $IM$ by $\alpha$, that is 
\begin{equation}\label{alpha}
\alpha :=\omega_1^1=\mathcal{J}\eta^1-\mathcal{I}\,\eta^3,
\end{equation}
consequently, considering the formulae  \eqref{structure eq. 1}, \eqref{structure eq. 3} and \eqref{Bianchi} together with \[d \mathcal{ I} = H(\mathcal{I }) \,\eta^1 + \mathcal{J }\,\eta^2 + V(\mathcal{I })\,\eta^3 , \quad d \mathcal{J } = H(\mathcal{ J}) \,\eta^1 + S(\mathcal{ J})\,\eta^2 + V(\mathcal{ J})\,\eta^3,\] 
we obtain \vspace{-0.15cm}
\begin{equation}\label{d_alpha}
d \alpha :=  \Omega^1_1 = V(\rho)\,\eta^1\wedge\eta^2- \mathcal{F}\,\eta^1\wedge\eta^3.
\end{equation}

\begin{obs}
If the spray $S$ is metrizable by a Riemannian function, thereby the main scalar $\mathcal{I}$ identically  vanishes, then the  geometric quantities $\mathcal{F}, \, \alpha $ and $d \alpha$  are vanishing. 
If $\alpha$ vanishes on the horizontal (respectively, vertical) distribution, then $(M,F)$ is Landsbregian (respectively,  Riemannian).  Also, $\ker(\alpha):=\{ X \in \mathfrak{X}(\T M) \,|\, \alpha(X)=0\}= \textit{\rm span}\{S\} $ on $IM$, i.e., $\ker(\alpha)$ is an invariant distribution by the Lagrangian planes. 
\end{obs}
 In what follows, we are going to study the Finsler surfaces which satisfy the aforementioned flag curvature (Ricci scalar) conditions. We start with $K$-basic Finsler surfaces.
 
\subsection{The flag curvature satisfies $V(K) =0$}

\begin{lema}\label{v(k)=0 is sectional curvature}
Let $S$ be the geodesic spray of a Finsler function $F$ and let $\mathcal{D}$ be its
Berwald distribution. Then, $V(\rho)=0$ if and only if  the flag curvature $K$ is a function of $x$ only, that is, the Finsler function is $K$-basic. 
\end{lema}
\begin{proof}
Since the Ricci scalar  $\rho= K\, F^2$ and $V(F)=0$, which follows from \eqref{V(F)=0},  thus \[V(\rho)=0  \text{ if and only if  }V(K)=0. \] Clearly,  due to  $\mathcal{C}(K)=0$, we get $  d_{J}K =V(K)\eta^1 .$    Therefore, the assumption $ V(K)=0$ reads that $ d_{v}K =0$.
Hence we have
$$ V(\rho)=0 \Longleftrightarrow  V(K)=0 \Longleftrightarrow   d_{J}K=0 \Longleftrightarrow d_{v}K =0 \Longleftrightarrow K \in C^{\infty}(M).$$
\vspace*{-0.95 cm}
\[\qedhere\]
\end{proof}
\begin{prop}\label{sf vanish}
Let $S$ be the geodesic spray of a Finsler function $F$ and let $\mathcal{D}$ be its
Berwald distribution. If $V(\rho)=0$, then  the following are satisfied:
\begin{enumerate}
\item{ The function $\mathcal{F}$ is a first integral of the geodesic flow, that is $S(\mathcal{F})=0$.}
\item{ $S(\mathcal{J})=-\mathcal{I}\, \rho$.}
\end{enumerate}
\end{prop}
\begin{proof}
It follows directly from the identities  \eqref{Bianchi} and \eqref{Bianchif}.
\end{proof}

\begin{prop} \label{vanishing of K in diection of v}
Let $S$ be the geodesic spray of a Finsler function $F$ and let $\mathcal{D}$ be its
Berwald distribution. Then, the following assertions are equivalent:
\begin{enumerate}

\item {The $1$-form $\alpha$ is invariant by the geodesic flow, that is $\mathcal{L}_{S}\alpha =0$.}

\item{ The Finsler function is $K$-basic.} 

\item{ The $1$-form $\mathcal{L}_{S}\alpha$ is closed.}
\end{enumerate}
\end{prop}
\begin{proof}
$(1) \Longleftrightarrow (2):$ 
It follows from the formulae \eqref{alpha} and \eqref{d_alpha} along with the fact that $\eta^2$ is the dual of $S$. Indeed,  in the view of \eqref{interior derv.} both $i_{S} (\mathcal{J}\eta^1-\mathcal{I}\,\eta^3 ) $ and $i_{S} (\eta^1\wedge\eta^3 ) $ identically vanish. Thus, we have 
\[\mathcal{L}_{S}\alpha = i_{S} d\alpha + d \, i_{S} \alpha = i_{S} ( -V(\rho)\eta^2\wedge\eta^1- \mathcal{F}\,\eta^1\wedge\eta^3 ) + d \, i_{S} (\mathcal{J}\eta^1-\mathcal{I}\,\eta^3 )  = i_{S} (V(\rho)\, \eta ^ 1 \wedge \eta ^{2}) = V(\rho)\, i_{S} (\eta ^ 1 \wedge \eta ^{2}).\]
That is, \vspace{-0.2cm}
\begin{equation}\label{ls alpha}
\mathcal{L}_{S}\alpha = V(\rho)\, \eta ^ 1.
\end{equation}
Then, we obtain $\mathcal{L}_{S}\alpha =0 \Longleftrightarrow V(K)=0 $ by Lemma~\ref{v(k)=0 is sectional curvature}.\\
$(2) \Longrightarrow (3):$ Since $d^2 =0$, we have
\[\mathcal{L}_{S} d\alpha = (i_{S} d + d i_{S})\,d\alpha = d (i_{S} \alpha ) = d(V(\rho)\, \eta ^ 1)= d(V(\rho))\, \wedge \eta ^ 1 + V(\rho)\,  d\eta ^ 1.\]
Substituting by the identity\eqref{Bianchif},  the formula \eqref{structure eq. 1} and $d(V(\rho))= H(V(\rho))\, \eta^1 + S(V(\rho))\, \eta^2 + V(V(\rho))\, \eta^3 ,$ we obtain
\begin{equation}\label{lsd alpla}
\mathcal{L}_{S} \,d\alpha= S(V(\rho))\, \eta^2\wedge\eta^1 +V(\rho)\,\eta^2\wedge\eta^3 +S(\mathcal{F}) \,\eta^1\wedge\eta^3.
\end{equation}
Hence, the postulate $V(K)=0$ gives $S(\mathcal{F})=0$ by Proposition~\ref{sf vanish}. Therefore, $(M,F)$ is $K$-basic Finsler surface implies that $\mathcal{L}_{S} \,d\alpha =0 $. \\
$(3) \Longrightarrow (2):$ Suppose that $\mathcal{L}_{S} \,d\alpha =0$, that is $V(K)=0$ and $S(\mathcal{F})=0$ by using \eqref{lsd alpla}. In fact, Proposition~\ref{sf vanish} says that $V(K)=0$ implies  $S(\mathcal{F})=0$. Thereby, we can say that $\mathcal{L}_{S} \,d\alpha =0$ if and only if  $V(K)=0$.
\end{proof}

Using Lemma~\ref{v(k)=0 is sectional curvature}, we can reformulate Theorem~\ref{Foulon result} as follows.
\begin{thm} \label{Th:V(K)=0}
Let $S$ be the geodesic spray of a Finsler function $F$ and let $\mathcal{D}$ be its
Berwald distribution. If $(M,F)$ is a connected compact Finsler surface of genus at least one without conjugate points such that $V(K)=0$, then $(M,F)$ is Riemannian.
\end{thm}
\begin{proof}
Assume that $V(K)=0$, then   by   Proposition~\ref{sf vanish},  $S(\mathcal{F})=0$. 
  Applying~\cite[Theorem 1]{Finslersur}, we get $\mathcal{F}$ is constant  on $IM$, \vspace{-0.2 cm}
that is $$d\mathcal{F}= H(\mathcal{F}) \eta^{1} + S(\mathcal{F}) \eta^{2}+V(\mathcal{F}) \,\eta^{3}=0.$$ Consequently, $H(\mathcal{F})=V(\mathcal{F})=0$. Now from Stokes Theorem, we have
$$\int d(\alpha \wedge \eta^{2}) =\int d\alpha \wedge \eta^{2} - \int \alpha \wedge d\eta^{2} =0.$$
Using the expressions of $ \alpha,\,d\alpha$, see \eqref{alpha},  \eqref{d_alpha}, and $d\eta^2$ from \eqref{structure eq. 2}, we obtain 
\[\int d\alpha \wedge \eta^{2} =\int \mathcal{F}\, \eta^{1}\wedge \eta^{2} \wedge \eta^{3},\,\, \quad\int \alpha \wedge d\eta^{2} = \int (\mathcal{J}\,\eta^{1} - \mathcal{I}\,\eta^{3}) \wedge \eta^{3} \wedge \eta^{1}=0. \]
Therefore, \vspace{-0.3cm}
$$\int \mathcal{F}\, \eta^{1}\wedge \eta^{2} \wedge \eta^{3} =0 \Longleftrightarrow \mathcal{F}\, \rm{Vol}(IM)=0,$$
where $\eta^{1}\wedge \eta^{2} \wedge \eta^{3}$ is the volume form of $IM$.  Hence, $\mathcal{F}=0$ on $IM$, accordingly, $d \alpha =0$ by \eqref{d_alpha}. So that, $\alpha$ is an exact $1$-form by Poincare lemma. In other words, there exists a function $\mu$ such that $\alpha= d \mu$, where \vspace{-0.11 cm}
\[d \mu = S(\mathcal{I})\,\eta^1 - \mathcal{I}\,\eta^3 = H(\mu) \eta^{1} + S(\mu) \eta^{2}+V(\mu) \eta^{3} . \]
Thus,  $ H(\mu)=S(\mathcal{I}),\,\,\,S(\mu)=0$ and $V(\mu)=-\mathcal{I}$.
Hence, $\mu$ is a first integral of the geodesic flow. Applying~\cite[Theorem 1]{Finslersur}, yields the function $\mu$ is constant on $IM$.  Consequently, $d \mu =0$, that is $H(\mu)=0$ and $V(\mu)=0$, which means $\mathcal{I}=0$ and $S(\mathcal{I})=0 $. Therefore, $(M,F)$ is Riemannian.
\end{proof}

\medskip 

\subsection{The flag curvature satisfies $V(K)=-\mathcal{I}/F^2$} 

\noindent
 
\medskip

It may be recalled that the condition $V(K)= -\mathcal{I}/F^2$ is equivalent to $ V(\rho)=~ -\mathcal{I}$. 

\begin{lema}\label{VK=I}
Let $S$ be the geodesic spray of a Finsler function $F$ and let $\mathcal{D}$ be its
Berwald distribution. If $V(\rho)=-\mathcal{I}$, then we have
\begin{enumerate}
\item {$S(\mathcal{F})=\mathcal{I}^2 + V(\mathcal{I}),\quad S(\mathcal{J})= \mathcal{I}(1-\rho) .$ } 
\item  {The $1$-form $\alpha$ is invariant by the geodesic flow if and only if $(M,F)$ is Riemannian.}
\end{enumerate}
\end{lema}
\begin{proof}
$(1)$ Substituting by $V(\rho) =-\mathcal{I}$ in  the identities \eqref{Bianchi} and \eqref{Bianchif}, we obtain the required. 

$(2)$ Plug the condition  $V(\rho) =-\mathcal{I}$ into the formula \eqref{ls alpha}, we get $\mathcal{L}_{S}\alpha = -\mathcal{I} \, \eta ^ 1$. Thus, $\mathcal{L}_{S}\alpha = 0 \Longleftrightarrow \mathcal{I}=0\,$ which completes the proof.
\end{proof}
\begin{thm}\label{Th:S(J)=0}
Let $S$ be the geodesic spray of a Finsler function $F$ and let $\mathcal{D}$ be its
Berwald distribution such that $V(\rho)=-\mathcal{I}$. If  either $S(\rho)=0$ or $S(\mathcal{J})=0$, then $(M,F)$ is Riemannian.
\end{thm}
\begin{proof}
Suppose that $S(\mathcal{J})=0$. By Lemma~\ref{VK=I}~(1),   we have $S(\mathcal{J})=0,$ which is equivalent to $\mathcal{I}\,(1-\rho)=0.$ That is, $\,\mathcal{I}=0$ or $(1-\rho)=0.$ Therefore, $\rho=1$ implies $V(\rho)=0$ and $V(\rho)=-\mathcal{I}$ gives  $\mathcal{I}=0$. That is, $(M,F)$ is Riemannian.

Applying the commutation formulae \eqref{Berwald_f2 SH} and \eqref{Berwald_f2 VS} to $\rho$, we obtain 
\[
S(H(\rho)) - H(S(\rho))= \rho\, V(\rho), \quad\,
S(V(\rho))- V(S(\rho))= - H(\rho),\\
\]
Now, assume that $S(\rho)=0$ and $V(\rho)=-\mathcal{I} $, we get
 $$ S(H(\rho))= -\mathcal{I}\,\rho ,\,\,\quad\,\,\, H(\rho)=\mathcal{J}.$$
  Therefore,  $S(\mathcal{J})= -\mathcal{I}\rho $. By Lemma~\ref{VK=I}~(1), we have $S(\mathcal{J})= \mathcal{I}\,(1-\rho).$ Hence,  $-\mathcal{I}\,\rho= \mathcal{I}(1-\rho),$ that is, $\mathcal{I}=0$. 
\end{proof}
\begin{cor} \label{cor2}
Let $S$ be the geodesic spray of a Landsbergian structure $F$ and let $\mathcal{D}$ be its
Berwald distribution such that $V(\rho)=-\mathcal{I} $, then $(M,F)$ is Riemannian.
\end{cor}
\begin{proof} 
 Suppose that $\mathcal{J}=0$. Consequently, $S(\mathcal{J})=0$. Now, by  Theorem~\ref{Th:S(J)=0}, $(M,F)$ is   Riemannian.
\end{proof}
It is worth mentioning that Theorem~\ref{Th:S(J)=0} and Corollary~\ref{cor2} are related to Theorems~\ref{Fsurconjresults} and \ref{Ikeda thm}.
\medskip
\vspace{-0.2cm}
\subsection{The flag curvature satisfies $V(K)=-\mathcal{I}\,K$}
\noindent
 
\medskip 

It should be noted that, the condition $V(K)= -\mathcal{I}\,K$ is equivalent to $ V(\rho)= -\mathcal{I}\,K$. 
\begin{thm}\label{thm1:s(J)=0}
Let $S$ be the geodesic spray of a Finsler function $F$ and let $\mathcal{D}$ be its
Berwald distribution. If $(M,F)$ is a connected compact Finsler surface of genus at least one without conjugate points such that $V(\rho)=-\mathcal{I}\, \rho$, then $(M,F)$ is Riemannian.
\end{thm}
\begin{proof}
The identity \eqref{Bianchi} when $V(\rho)=-\mathcal{I}\,\rho$ gives $S(\mathcal{J})=0$. Applying~\cite[Theorem 1]{Finslersur}, we get $\mathcal{J}$ is constant on $IM$. By Stokes Theorem, we have
$$\int d(\alpha \wedge \eta^{1}) =0 \Longleftrightarrow \int d\alpha \wedge \eta^{1} -\int \alpha \wedge d \eta^{1}=0.$$
Using the expressions of $ \alpha,\,d\alpha$  and $d\eta^1$ see \eqref{alpha},  \eqref{d_alpha} and \eqref{structure eq. 1}, we obtain 
\[\int d\alpha \wedge \eta^{1} = \int V(\rho)\, \eta^{1}\wedge \eta^{2} \wedge \eta^{1} -\int \mathcal{F}\, \eta^{1}\wedge \eta^{3} \wedge \eta^{1} =0,\]
\[ \int \alpha \wedge d\eta^{1} = \int (\mathcal{J}\,\eta^{1} - \mathcal{I}\,\eta^{3}) \wedge (\eta^{2} \wedge \eta^{3} - \mathcal{I}\,\eta^{1} \wedge \eta^{3}) =  \int \mathcal{J}\,\eta^{1}\wedge \eta^{2} \wedge \eta^{3} . \]
Consequently, we get
$$\int \mathcal{J}\, \eta^{1}\wedge \eta^{2} \wedge \eta^{3} =0 \Longleftrightarrow \mathcal{J}\, \rm{Vol}(IM)=0.$$ Hence, $\mathcal{J}=0$ on $IM$, that is,  $(M,F)$ is Landsbergian. Applying Theorem~\ref{Fsurconjresults}~(1), we get $(M,F)$ is Riemannian. 
\end{proof}

\begin{prop}
Let $S$ be the geodesic spray of a Finsler function $F$ and let $\mathcal{D}$ be its
Berwald distribution such that $V(\rho)=-\mathcal{I}\, \rho$ and $S(\rho)=0$, then the trace of $hv$-Berwald curvature identically vanishes.
\end{prop}
\begin{proof}
Applying the Lie brackets of Berwald frame \eqref{Berwald_f2 VS} and \eqref{Berwald_f2 HV} to $\rho$, yields
\[
S(V(\rho))- V(S(\rho))= - H(\rho),\quad \,\,
H(V(\rho)) -V(H(\rho))= S(\rho) +\mathcal{I}\,H(\rho) +\mathcal{J} \, V(\rho).
\]
Substituting by $V(\rho)=-\mathcal{I}\, \rho$ and $S(\rho)=0$ in the last two equations above,  gives
$$H(\rho)=\mathcal{J}\,\rho,\,\quad\,H(\mathcal{I})\rho + 2\, \mathcal{I}\,H(\rho) + V(H(\rho))= \mathcal{I}\,\mathcal{J}\,\rho.$$ 
Plug $H(\rho)=\mathcal{J}\,\rho$ into $H(\mathcal{I})\,\rho + 2\, \mathcal{I}\,H(\rho) + V(H(\rho))= \mathcal{I}\,\mathcal{J}\,\rho$, we obtain
\begin{eqnarray*}
\mathcal{I}\,\mathcal{J}\,\rho &=& H(\mathcal{I})\,\rho + 2 \,\mathcal{I}\,\mathcal{J}\,\rho + V(\mathcal{J}\,\rho) \\
 &=& H(\mathcal{I})\,\rho +  2 \,\mathcal{I}\,\mathcal{J}\,\rho + V(\mathcal{J})\,\rho -  \mathcal{J}\,\mathcal{I}\,\rho.
\end{eqnarray*}
That is, $H(\mathcal{I})\,\rho + V(\mathcal{J})\,\rho =0 $ which is equivalent to $\rho\,\mathcal{F}=0$. Thereby,  $\rho=0$ or  $\mathcal{F}=0$. The proof is completed by looking at Lemma~\ref{trace of Berwald cuvature} and considering $\rho \neq 0$ as our Finsler spray is nonflat.
 \end{proof}

The following result can be considered as a generalization of Theorem~\ref{Fsurconjresults}~(3), this is because of a Finsler surface is not necessarily compact of genus at least one without conjugate points. 

\begin{thm}\label{thm:s(J)=0}
Let $S$ be the geodesic spray of a Finsler function $F$ and let $\mathcal{D}$ be its
Berwald distribution such that $V(\rho)=-\mathcal{I}\, \rho$ and $S(\rho)=0$, then $(M,F)$ is Riemannian.
\end{thm}
\begin{proof}
Applying the commutation formulae \eqref{Berwald_f2 SH} and \eqref{Berwald_f2 VS} to $\rho$, yields
\[
S(H(\rho)) - H(S(\rho))= \rho\, V(\rho), \quad 
S(V(\rho))- V(S(\rho))= - H(\rho).\]
Pluging  $V(\rho)=-\mathcal{I}\, \rho$ and $S(\rho)=0$ into the last equations above, we obtain
 \[ S(H(\rho))= -\mathcal{I}\,\rho^2 ,\,\,\,\,\,\, \mathcal{J}\,\rho =H(\rho).\]
 Therefore, \vspace{-0.3cm}
 $$S(H(\rho))=S(\mathcal{J}\,\rho)= -\mathcal{I}\,\rho^2  \Longleftrightarrow \rho\,S(\mathcal{J}) + \mathcal{J}\, S(\rho)= -\mathcal{I}\,\rho^2.$$
But we have $S(\mathcal{J}) =0$ and $S(\rho)=0$, which gives $\mathcal{I}\,\rho^2 =0.$ 
Therefore,  $\mathcal{I} =0$  due to the spray $S$ is nonflat.
\end{proof} 

\textbf{Acknowledgment.}
I would like to express my deep gratitude to Professor I. Bucataru (Alexandru Ioan Cuza University, Romania) and Dr.~S.~G. Elgendi (Benha University, Egypt) for their useful discussions and  comments.


\end{document}